\documentclass[10pt,journal,letterpaper,twoside,nofonttune]{IEEEtran}

\usepackage{mathpazo}
\usepackage{times}

\usepackage{amsmath}  
\usepackage{amsfonts}
\usepackage{amssymb}  
\usepackage{mathrsfs} 

\usepackage{theorem}  
\usepackage{cite}     
 
\usepackage{upref}

\usepackage{comment}  
\usepackage{enumerate}

\usepackage{color}

\title{\Huge 
Asymptotic Improvement of the Gilbert-Varshamov Bound
on the Size of Permutation Codes\\[0.50ex]}

\author{\large%
Michael Tait,
Alexander Vardy, 
and Jacques Verstra\"ete
\vspace{-2.50ex}
\thanks{Michael Tait and Jacques Verstra\"ete are with the
   Depart\-ment of Mathematics,
   University of California San Diego,
   La Jolla, CA 92093, U.S.A.
   (e-mail: \texttt{mtait@math.ucsd.edu}, \texttt{jacques@ucsd.edu}).}
\thanks{Alexander Vardy is with the
   Department of Electrical and Computer Engine\-ering,
   the Department of Computer Science and Engineering,
   and the Depart\-ment of Mathematics,
   University of California San Diego,
   La Jolla, CA 92093, U.S.A.
  (e-mail: \texttt{avardy@ucsd.edu}).}
}

\date{}

\theoremstyle{plain}
\theorembodyfont{\normalfont\slshape}

\newtheorem{thm}{Theorem\hspace{-1pt}}
\newenvironment{theorem}
{\begin{thm}\hspace*{-1ex}{\bf.}}{\end{thm}}

\newtheorem{lem}[thm]{Lemma\hspace{-.75pt}}
\newenvironment{lemma}{\begin{lem}\hspace*{-1ex}{\bf.}}{\end{lem}}

\newtheorem{prop}[thm]{Proposition$\!$}

\newtheorem{cor}[thm]{Corollary$\!$}
\newenvironment{corollary}{\begin{cor}\hspace*{-1ex}{\bf.}}{\end{cor}}

\newtheorem{defn}{Definition$\!$}

\newcounter{enumrom}
\renewcommand{\theenumrom}{(\roman{enumrom})}

\makeatletter
\renewcommand{\@endtheorem}{\endtrivlist}
\makeatother

\makeatletter
\renewcommand{\thefigure}{{\bf \@arabic\c@figure}}
\renewcommand{\fnum@figure}{{\bf Figure}\,\thefigure}
\makeatother

\newcommand{\cG}{{\cal G}}

\newcommand{\cS}{{\cal S}}

\DeclareMathAlphabet{\mathbfsl}{OT1}{ppl}{b}{it} 

\newcommand{\ceil}[1]{\left\lceil #1 \right\rceil}

\newcommand{\be}[1]{\begin{equation}\label{#1}}
\newcommand{\ee}{\end{equation}}
\newcommand{\eq}[1]{(\ref{#1})}

\renewcommand{\le}{\leqslant}
\renewcommand{\leq}{\leqslant}
\renewcommand{\ge}{\geqslant}
\renewcommand{\geq}{\geqslant}

\renewcommand{\Bbb}{\mathbb}

\newcommand{\R}{{\Bbb R}}

\newcommand{\dfn}{\bfseries\itshape}

\newcommand{\Cref}[1]{Co\-ro\-lla\-ry\,\ref{#1}}

\newcommand{\deff}{\mbox{$\stackrel{\rm def}{=}$}}

\begin{document}

\maketitle

\begin{abstract}
Given positive integers $n$ and $d$, let $M(n,d)$ denote the maximum
size of a permutation code of length $n$ and minimum Hamming distance
$d$.  The Gilbert-Varshamov bound asserts that $M(n,d) \geq n!/V(n,d{-}1)$
where $V(n,d)$ is the volume of a Hamming sphere of radius $d$ in $\cS_n$.
Recently, Gao, Yang, and Ge showed that this bound can be improved
by a factor $\Omega(\log n)$, when $d$ is fixed and $n \,{\to}\, \infty$.
Herein, we consider the 
situation where the ratio
$d/n$ is fixed and improve the Gilbert-Varshamov bound by
a~factor that is \emph{linear in $n$}. That is, we show that
if $d/n < 0.5$, then
$$
M(n,d)
\: \geq \
cn\,\frac{n!}{V(n,d\,{-}\,1)}
$$
where $c$ is a positive constant that depends
only on $d/n$.  To establish this result, we follow
the method of Jiang and Vardy. Namely, we recast
the problem of bounding $M(n,d)$ into a graph-theoretic framework
and prove that the resulting graph is locally sparse.
\end{abstract}

\begin{keywords}
Ajtai-Koml\'os-Szemer\'edi bound,
Gilbert-Var\-sha\-mov bound,
locally sparse graphs,
permutation codes
\vspace{1.00ex}
\end{keywords}

\section{Introduction}

\noindent
\PARstart{L}{et} $\cS_n$ be the symmetric group of permutations on $n$
elements. A \emph{\dfn permutation code} $C$ is a subset of $\cS_n$.
For $\sigma,\tau$ in~$\cS_n$, the \emph{\dfn Hamming
distance} $d(\sigma,\tau)$ between them is the number of
positions where they differ, namely:
\be{distance-def}
d(\sigma,\tau) \,\ \deff\,\ \bigl\{i\in [n] ~:~ \sigma(i) \not = \tau(i)\bigr\}
\ee
We say that a permutation code $C \subseteq \cS_n$ has minimum
distance $d$ if each pair of permutations in $C$ is at Hamming distance
at least $d$. The maximum number of codewords in a permutation code of
minimum distance $d$ will be denoted $M(n,d)$.

Permutation codes were first investigated in
\cite{Blake,BCD,DV,DF,Slepian}. In recent 
years, they have received considerable
renewed attention, due to their application
in data transmission over power lines \cite{FV,V,HPVY,CKL,CCD}.
Other applications of permutation codes include
design of block ciphers~\cite{CLT}
and coding for flash memories~\cite{BM,BJMS,BJS}
(although the latter application involves the Kendall
$\tau$-metric in lieu of the Hamming distance).

A key problem in the theory of permutation codes is to determine
$M(n,d)$. Various bounds on $M(n,d)$ have been proposed in
\cite{CKL,GGY,DFKW,DV,DS,FK,MS2}.
For small values of $d$, computer search has been used to
find exact values of $M(n,d)$ in~\cite{MS2}.
However, in general, the problem
is very difficult, and little progress has been made for $d > 4$.

For $x \in \R$, let $[x]$ denote the operation of rounding to
the~ne\-arest integer. Then the number of derangements of $k$
elements is given by $D_k = [n!/e]$. Under the Hamming distance
defined in \eq{distance-def}, a sphere of radius $r$ has volume
$$
V(n,r)
\,\ \deff\
\sum_{k=0}^r {n\choose{k}} D_k
$$
The 
Gibert-Varshamov bound \cite{G,Var}
and the sphere-packing bound~\cite{MS} now give the following.
\begin{theorem}
\be{GV-def}
\frac{n!}{V(n,d{-}1)}
\ \leq \
M(n,d)
\ \leq \
\frac{n!}{V\bigl(n,\lfloor\!\frac{d{-}1}{2}\!\rfloor\bigr)}
\vspace{1.00ex}
\ee
\end{theorem}

The Gilbert-Varshamov bound is used extensively in
coding theory \cite{MS,PH}. As is well known, improving this bound
asymptotically is a difficult task~\cite{PH}.

Recently, Gao, Yang, and Ge~\cite{GGY} showed that
the Gilbert-Varshamov bound in \eq{GV-def} can be improved
by a factor $\Omega(\log n)$, when $d$ is fixed and $n \,{\to}\, \infty$.
In this paper, we complement this work, focusing on the more natural
case where $d/n$ is a~fixed ratio 
and $n$ tends to infinity. Our main theorem is the following.
\begin{theorem}
\label{main}
Let ${d}/{n}$ be a fixed ratio with\, $0 \,{<}\: d/n \:{<}\, 0.5$.~Then
as $n\to \infty$, we have
\be{eqmain}
M(n,d)
\:\geq\
\Omega\left( \frac{n!}{V(n,d{-}1)}\, n \right)
\vspace{1.00ex}
\ee
\end{theorem}

To prove this theorem, we follow the method of~\cite{JV}.
Namely, we will construct a graph in which every
independent set corresponds to 
a permutation
code of minimum distance $d$.  Using\linebreak techniques from graph theory, we
will then
obtain a lower bound on the size of the largest independent set
in this graph.

\section{Graph Theory}

\noindent\looseness=-1
Our graph notation is standard; the reader is referred to
\cite{W} for any undefined terms.  A {\em graph} $G$ consists of a set
of {\em vertices} $V(G)$ and a set of {\em edges} $E(G)$, which are
pairs of vertices. We assume throughout that the sets $V(G)$ and
$E(G)$ are finite.
Two vertices $u$ and $v$ are {\em adjacent\/} if
$\{u,v\}\in E(G)$. For a subset $S$ of $V(G)$,
the {\em subgraph induced by $S$} is the graph that
has vertex set $S$ and edge set $\bigl\{\{u,v\}\in E(G): u,v\in S\bigr\}$.
The~{\em neighbor\-hood\/} of a vertex $v$ is the set
of all vertices adjacent to $v$, and the {\em neighborhood graph of $v$}
is the subgraph induced by the
neighborhood of $v$.  The {\em degree} of a vertex $v$ is the size of the
neighborhood of $v$.  We let $\Delta(G)$ denote the
{\em maximum vertex degree} in a graph $G$.
A set $K\,{\subseteq}\,V(G)$ is a {\em clique} if the subgraph induced
by $K$ has every possible edge, and a set $I\,{\subseteq}\,V(G)$ is an
{\em independent set\/} if the subgraph induced by $I$ has no edges.
A~{\em tri\-angle\/} is a clique of size $3$. The maximum number of vertices
in an independent set of $G$ is called the {\em independence~num- ber\/}
and denoted by $\alpha(G)$.

Consider the graph $\cG$ whose vertex set is $\cS_n$ with two
permutations $\sigma,\tau$ being adjacent
if and only if $1\leq d(\sigma, \tau) < d$.
Then it is clear that any permutation code of length $n$ and
minimum distance $d$ is an independent set in $\cG$. Conversely,
any independent set in $\cG$ is a permutation code of length $n$ and
minimum distance $d$.  This bijection proves the following.
\begin{lemma}
\label{bijection}
$M(n,d) = \alpha(\cG)$.
\end{lemma}

Observe that every vertex in $\cG$ has the same degree,
given by
$\sum_{k=1}^{d-1} {n\choose{k}}D_k$.
It is graph theory folklore that for any graph $G$,
\begin{equation}
\label{chromaticBound}
\alpha(G) \:\geq\: \frac{|V(G)|}{\Delta(G)+1}
\end{equation}
It is evident that \eqref{chromaticBound}
along with Lemma\,\ref{bijection} immediately recovers
the Gilbert-Varshamov bound in \eq{GV-def}.
We note that we have recovered this bound
using very little information about $\cG$.

Our strategy to prove
Theorem\,\ref{main} is to use the relative sparseness of the~neigh\-borhood
graph of every vertex in $V(\cG)$~in~order to strengthen \eqref{chromaticBound}.
This technique has been introduced in \cite{JV} to improve the
Gilbert-Varshamov bound on the size of binary codes.
It was later applied to improve lower bounds on
the size of $q$-ary codes~\cite{VW} and sphere packings~\cite{KLV}.

We will need some results about locally sparse graphs. Ajtai,
Koml\'os, and Szemer\'edi~\cite{AKS} showed that one can improve
\eqref{chromaticBound} if $G$ has no triangles.
The following lemma is from~\cite{AKS} (but see also
\cite[p.\,272]{AS} for a much shorter proof of the same result).
\begin{lemma}
Let $G$ be a graph with maximum degree $\Delta$, and suppose that $G$
has no triangles. Then
\begin{equation}
\alpha(G)
\,\geq\, \frac{|V(G)|}{8\Delta}\log_2 \Delta
\end{equation}
\end{lemma}

This result was extended in \cite[Lemma\,15, p.\,296]{B}
from graphs with no triangles to graphs with relatively few triangles.

\begin{lemma}\label{triangleBound}
Let $G$ be a graph with maximum degree $\Delta$ and suppose that $G$
has at most $T$ triangles.  Then
\begin{equation}
\alpha(G)
\,\geq\,
\frac{|V(G)|}{10\Delta}
\left(\ln \Delta \,-\,
\frac{1}{2}\ln\! \left(\frac{T}{|V(G)|}\right)\right)
\vspace{1.00ex}
\end{equation}
\end{lemma}

A graph has no triangles if and only if the neighborhood of every
vertex is an independent set, and a graph has relatively few triangles
if the neighborhoods of its vertices are relatively sparse.  We
make this precise in the following corollary.

\begin{corollary}\label{cor1}
Let $G$ be a graph with maximum degree $\Delta$~and~suppose that for
all $v\,{\in}\,V(G)$, the subgraph induced by the~neighborhood of\/ $v$ has
at most $E$ edges.  Then
\begin{equation}
\alpha(G)
\:\geq\:
\frac{|V(G)|}{10\Delta}
\left(\ln \Delta \,-\, \frac{1}{2}\ln\!\left(\frac{E}{3}\right)\right)
\end{equation}
\end{corollary}
\begin{proof}
The number of triangles incident with a vertex $v$~is~eq\-ual
to the number of edges in the subgraph induced by the neighborhood of
$v$.  Therefore, for every $v\,{\in}\,V(G)$, there are at most $E$
triangles incident with $v$.  Summing over the vertex set of $G$ and
noting that each triangle is incident with exactly $3$ vertices gives
the result.
$\blacksquare$
\end{proof}

\section{Proof of the Main Theorem}

\noindent
In order to use Corollary\,\ref{cor1}, we must count the number
of edges in the neighborhood of every vertex in $V(\cG)$. To do so,
first~note that $\cG$ is vertex transitive because for all
$\sigma,\tau \in \cS_n$, we have
$$
d(\sigma, \tau)
\,=\,
d(id, \sigma \tau^{-1})
$$
where $id$ is the identity element of $\cS_n$.
This implies that the number of edges in
the neighborhood of a vertex $v\,{\in}\,V(\cG)$~does not depend on $v$.
Therefore, to simplify the calculation,
we will consider only the neighborhood of the identity
permutation.

By Corollary\,\ref{cor1},
in order to prove our main result in \eqref{eqmain},
it would suffice to show that
\begin{equation}
\label{bound}
\log\left(\frac{\Delta^2}{E}\right)
\,=\
\Omega(n)
\end{equation}
where $E$ is the number of edges in the subgraph induced by the
neighborhood of the identity, and $\Delta$ is given by
\be{Delta-def}
\Delta \ \ \deff\ \sum_{k=1}^d {n\choose{k}} D_k
\ee
Note that, for notational convenience, 
the sum in \eqref{Delta-def} extends up to $d$ rather $d-1$.
Since ${d}/{n}$ is fixed, this does not matter.

It follows from \eq{bound}
that the proof of Theorem\,\ref{main}
reduces to estimating $E$ and
considering the asymptotics of the ratio $\Delta^2/E$.
We count the number of edges in the subgraph induced
by the neighborhood of the identity as follows.

\begin{enumerate}
\item
Vertices in the neighborhood graph of the identity
have~distance between $1$ and $d-1$ from the identity. 
We will count the edges incident with permutations
$\sigma$ that are at distance $s$ from $id$,
and then sum over $s = 1,2,\ldots,d{-}1$. Note that,
given a fixed distance $s$, there are exactly
${n\choose{s}}D_s$ permutations at distance $s$
from the identity.

\item
Let us fix a permutation $\sigma$ at distance $s$ from
the identity and count how many edges in the neighborhood
graph of the identity are incident with $\sigma$. To do this, we
will sum over permutations $\tau$ that are at distance $t$
from the identity, for $t = 1,2,\ldots,d{-}1$. Note that
for $\tau$ to be incident with~$\sigma$, we must
also have $d(\tau,\sigma) < d$.

\item
To count how many permutations $\tau$ satisfy the above requ\-irements,
we let $m$ be the number of indices where $\sigma$~and~$\tau$
are deranging the same position. That is, given $\sigma$ and $\tau$, let
\begin{equation}
\label{mBound}
m \ \ \deff\ \ \bigl|\{i\in [n] ~:~ \sigma(i) \not= i, \tau(i)\not = i\}\bigr|
\end{equation}
\looseness=-1
By assumption, $\sigma$ is deranging $s$ positions while $\tau$
is deranging $t$ positions. Thus $m \le \min(s,t)$. Also notice
that in the $s - m$ positions where $\sigma$ is deranging but $\tau$
is not, the two permutations necessarily differ. Same goes for the
$t - m$ positions where $\tau$ is deranging but $\sigma$ is not.
Hence 
\begin{equation}
\label{mBound2}
(s-m) \,+\, (t-m) \,\le\, d(\sigma,\tau) \,<\, d
\end{equation}
It follows that the parameter $m$ defined in \eq{mBound} is
in the range
$\lceil (s+t-d)/2\rceil^+\! < m \leq \mathrm{min}(s,t)$,
where $\lceil x\rceil^+$~denotes the smallest nonnegative
integer $k$ with $k\geq x$.

\item
Now let us suppose that $s,t,m$ are fixed, and count those
positions where $\sigma$ and $\tau$ are both deranging,
but still map to the same value. That is, let
\begin{equation}
\label{rBound}
r \ \ \deff\ \
\bigl|\{i\in [n]: \sigma(i) \ne i, \tau(i) \ne i, \sigma(i) = \tau(i)\}\bigr|
\end{equation}
Notice that $d(\sigma,\tau) = s+t-m-r$. It follows that the
parameter $r$ defined in \eq{rBound} satisfies
$r > \lceil s+t-m-d\rceil^+$.
\end{enumerate}
Our strategy for counting $E$ is to fix $s, t, m, r$ 
and count the number of pairs of permutations $\sigma$ and $\tau$
that have these parameters. We will then sum over $s, t, m, r$
in the appropriate ranges. 

As already noted above, there are exactly ${n \choose s}D_s$
permutations $\sigma$ at a given distance $s$ from the identity.
Given $\sigma$, there are ${s \choose m}$ ways to select the $m$
positions where both $\sigma$ and $\tau$ are deranging. Given
these $m$ positions, there are ${m \choose r}$ ways to pick
the $r$ positions where both permutations are deranging
but have the same image, as in \eq{rBound}. We have now chosen
our permutation $\sigma$, and part of our permutation $\tau$.
In particular, we have chosen $m$ positions where $\tau$
is deranging.
To specify the rest of $\tau$, we first choose the other
$t-m$ positions where $\tau$ is deranging. This
can be done in exactly ${n-s \choose t-m}$ ways. Now
that we have chosen the $t$ positions where $\tau$ is
deranging, we must pick the image of $t-r$ out of the $t$
positions (the image of $\tau$ is already specified as equal to the
image of $\sigma$ on $r$ positions). This can be done in less
than $(t-r)!$ ways. Note that we are overcounting here because
we must derange the $t-r$ positions and because there are some
restrictions on the indices where $\sigma$ and $\tau$ overlap.
However, this overcounting will not hurt our final bound.

To summarize the foregoing discussion, if we define
\be{g-def}
g(s,t,m,r)
\,\ \deff\,\
{n\choose{s}} D_s {s\choose{m}}{m\choose{r}} {n-s\choose{t-m}}(t-r)!
\ee
then we have that
\be{E-bound}
E
\ \leq \
\sum_{s=1}^d \sum_{t=1}^d
\sum_{m=\left\lceil \frac{s+t-d}{2} \right\rceil^+}^{\min\{s,t\}}
\sum_{r = \ceil{s+t-m-d}^+}^m \hspace{-1.5ex}g(s,t,m,r)
\ee
Our next task is to obtain an upper bound on 
$g(s,t,m,r)$~in~\eq{E-bound}. Indeed, if we can show that
$g(s,t,m,r) \le G$ whenever $s,t,m,r$ satisfy the constraints
\begin{align}
\label{constraint1}
\hspace*{-10ex}1\leq\,\, &s\leq d
\\
\label{constraint1a}
\hspace*{-10ex}1\leq\,\, &t\leq d
\\
\label{constraint2}
\hspace*{-10ex}\left\lceil \frac{s+t-d}{2}\right\rceil^+\!\!
\leq\,\, &m \leq \min\{s,t\}
\\
\label{constraint3}
\hspace*{-10ex}\lceil s+t-m-d\rceil^+ \leq\,\, &r\leq m
\end{align}
then we can conclude from \eq{E-bound} that $E \le d^4 G$. In fact,
since~we are interested only in the asymptotics of $E$ as $n \to \infty$,
an~asymptotic upper bound on $g(s,t,m,r)$ would suffice. The next
lemma shows that the value of $g(s,t,m,r)$ at $s = t = m = d$, $r=0$
can serve as such a bound.

\begin{lemma}
\label{lemma7}
Suppose that $s,t,m,r$ satisfy the constraints in equations\/
{\rm \eq{constraint1}\,--\,\eq{constraint3}}. Then for all
real $\varepsilon$
in the range $0 < \varepsilon < 1/6$ and for all sufficiently large $n$,
we have
\be{entropy-lemma}
\log_2\!\left( \frac{g(d,d,d,0)}{g(s,t,m,r)}\right)
\ \ge\ -3n h_2(3\varepsilon)
\ee
\end{lemma}

\begin{proof}
Let a positive $\varepsilon < 1/6$ be fixed, and recall
that $d = \delta n$ for a positive constant $\delta < 0.5$.\
We first show that \eq{entropy-lemma}~holds~vac- uously
unless $t \ge (1-\varepsilon)d$. Indeed,
\begin{equation}
\label{bound1}
\frac{g(d,d,d,0)}{g(s,t,m,r)}
\,\geq\,
\frac{{n \choose d} D_d d!}{{n \choose s} D_s {n\choose{n/2}}^3t!}
\,\geq\,
\frac{d!}{2^{3n} t!}
\end{equation}
It is easy to see that if $t \le (1-\varepsilon)d$, then the RHS of
\eq{bound1} grows without bound as $n\,{\to}\,\infty$. By a similar
argument, \eq{entropy-lemma}~holds~un\-less $s \ge (1-\varepsilon)d$.
We have
\begin{equation}
\label{bound1a}
\frac{g(d,d,d,0)}{g(s,t,m,r)}
\,\geq\,
\frac{{n \choose d} D_d d!}{{n \choose s} D_s {n\choose{n/2}}^3t!}
\,\geq\,
\frac{D_d}{2^{3n} D_s}
\end{equation}
which grows without bound as $n\,{\to}\,\infty$ if
$s \le (1-\varepsilon)d$.
We next show that \eq{entropy-lemma} again holds vacuously
unless $r \le \varepsilon t$. 
Indeed, if $r > \varepsilon t$, then
$$
\frac{g(d,d,d,0)}{g(s,t,m,r)}
\,\geq\,
\frac{{n\choose{d}}D_d d!}
{{n\choose{s}}D_s {n\choose{n/2}}^3 {((1-\varepsilon)t)!}}
\,\geq\,
\frac{d!}{2^{3n} ((1-\varepsilon)d)!}
$$
which grows without bound as $n\,{\to}\,\infty$.
It remains to consider the case where
\be{str-bounds}
s \ge (1-\varepsilon)d,
\hspace{1ex}
t \ge (1-\varepsilon)d
\hspace{4ex}
\text{and}
\hspace{4ex}
r \le \varepsilon d
\ee
Observe that
in this case, we must also have $m \geq (1-3\varepsilon)d$ in view
of \eq{constraint3}. Further observe that
\begin{equation}\label{bound3}
g(s,t,m,r)
\ \leq \
{n\choose{d}} D_d {d\choose {m}} {m\choose{r}} {n\choose{d-m}} d!
\end{equation}
for all $s,t,m,r$ satisfying the constraints
\eq{constraint1}\,--\,\eq{constraint3}. Combining this
with \eq{str-bounds} and $m \geq (1-3\varepsilon)d$,
we obtain
\begin{align}
g(s,t,m,r)~
& \leq \:
g(d,d,d,0)
\binom{d}{3\varepsilon d} \binom{d}{\varepsilon d} \binom{n}{3\varepsilon d}
\\[0.5ex]
& \leq \:
g(d,d,d,0) \binom{n}{3\varepsilon n}^{\!3}
\\[1.5ex]
\label{Lemma7-last}
& \leq \:
g(d,d,d,0) 2^{3n h_2(3\varepsilon)}
\end{align}
where
$h_2(x) = -x \log_2 \!x \,-\, (1-x) \log_2(1-x)$
is the binary entropy function. The lemma now follows from \eq{Lemma7-last}. $\blacksquare$
\vspace{.75ex}
\end{proof}

We are now in a position to
complete the proof of Theorem\,\ref{main}.
Recall that it would suffice to show that
$$
\log_2\! \left(\frac{\Delta^2}{E}\right) \,=\, \Omega (n)
$$
as in \eq{bound}, where $\Delta$ is given by \eq{Delta-def}
and $E$ is upper-bounded by~\eq{E-bound}. In order to simplify
the bound in \eq{E-bound}, let us define
$$
G(n,d)
\ \ \deff\
\max_{s,t,m,r} g(s,t,m,r)
$$
where the maximum is over $s,t,m,r$ satisfying the con\-straints
in \eq{constraint1}\,--\,\eq{constraint3}. Then
$E \le d^4 G(n,d)$.
As a lower bound on $\Delta$, we use
the largest term in the sum of \eq{Delta-def}. Thus
$$ 
\Delta \,\ge\, {n\choose d} D_d
$$ 
Combining this with 
Lemma\,\ref{lemma7}, we conclude that for all
$\varepsilon$ in the range $0 < \varepsilon < 1/6$
and for all sufficiently large $n$, we have
\begin{align}
\log_2\!\left(\frac{\Delta^2}{E}\right)
& \geq\,
\log_2\!\left(\frac{\binom{n}{d}^2 D_d^2}{d^4 G(n,d)}\right)
\nonumber\\
& =\,
\log_2\!\left(\frac{\binom{n}{d}^2 D_d^2}{d^4 g(d,d,d,0)}\right)
\,+\; \log_2\!\left(\frac{g(d,d,d,0)}{G(n,d)}\right)
\nonumber\\
\label{final-bound}
& \geq\,
\log_2\!\left(\frac{\binom{n}{d}^2 D_d^2}{d^4 g(d,d,d,0)}\right)
\,-\; 3n h_2(3\varepsilon)
\end{align}
Since $\lim_{\varepsilon \to 0} 3 h_2(3\varepsilon) = 0$,
we can conclude from \eq{bound} and \eq{final-bound}
that it remains to show
$$
\log_2\!\left(\frac{\binom{n}{d}^2 D_d^2}{d^4 g(d,d,d,0)}\right)
=\,
\log_2\!\left(\frac{\binom{n}{d}^2 D_d^2}{d^4 \binom{n}{d} D_d d!}\right)
=\,
\Omega(n)
$$
This is easily accomplished as follows:
\begin{align*}
\!\log_2\!\left(\frac{\binom{n}{d}^2 D_d^2}{d^4 \binom{n}{d} D_d d!}\right)
\ge &~ \log_2\! \left( \frac{\binom{n}{d}}{3d^4}\right)
\\
\ge &~ \,d\log_2\! \left(\frac{n}{d}\right) - 4\log_2 (d) - \log_2\!3
\\
= &~ \,\delta n\log_2\! \left(\frac{1}{\delta}\right)
- 4\log_2 (\delta n) - \log_2 \!3
\end{align*}
Since $\delta = d/n$ is a positive constant strictly less
than $0.5$,~we see that $\delta\log_2(1/\delta)$ is positive
and, hence, the expression above is $\Omega(n)$.
This completes the proof of Theorem\,\ref{main}.

\vspace{1.0ex}

\bibliographystyle{IEEEtran}

\end{document}